\documentclass[12pt]{article}
\usepackage{color,amsmath,amsfonts,amssymb,epsfig,latexsym}
\usepackage{graphicx}
\usepackage[latin1]{inputenc}
\usepackage{amssymb}
\usepackage{bm}
\usepackage{color}

\newtheorem{teo}{Theorem}[section]

\newenvironment{proof}[1][Proof]{\medskip\noindent\textit{#1. }\upshape}{\medskip}

\newtheorem{proposition}{Proposition}[section]
\numberwithin{equation}{section}
\begin{document}

\date{}
\title{Strong solutions in $L^2$ framework for fluid-rigid body interaction
problem - mixed case}
\author{{\large Hind Al Baba$^{2,4}$, Nikolai V. Chemetov$^1$, \v S\'arka
Ne\v casov\'a$^{2,3}$, Boris Muha$^5$} \\
\\
{\small $^1$ CMAF / University of Lisbon, Portugal }\\
{\small chemetov@ptmat.fc.ul.pt}\\
{\small $^2$ Institute of Mathematics, }\\
{\small \v Zitn\'a 25, 115 67 Praha 1, Czech Republic }\\
{\small $^3$ matus@math.cas.cz}\\
{\small $^4$ albaba@math.cas.cz}\\
{\small $^5$ Department of Mathematics}\\
{\small Faculty of Science}\\
{\small University of Zagreb,Croatia}\\
{\small borism@math.hr}}
\maketitle
\date{}

\section{Introduction}

In this paper we investigate the motion of a rigid body inside a viscous
incompressible fluid when mixed boundary conditions are considered. The
fluid and the body occupy a bounded domain $\mathcal{O}\subset\mathbb{R}^{d}$
($d=2$ or $3$).

In order to describe our approach, let us denote by $\mathcal{B}(t)\subset
\mathcal{O}$ a bounded domain occupied by the rigid body and a domain filled
by the fluid by $\mathcal{F}(t)=\mathcal{O}\setminus \overline{\mathcal{B}(t)%
}$ at a time moment $t\in \mathbb{R}^{+}.$\ \ Assuming that the initial
position $\mathcal{B}(0)$ of the rigid body is prescribed, for simplicity of
notation we denote $\mathcal{B}_{0}=\mathcal{B}(0)$ and $\mathcal{F}_{0}=%
\mathcal{F}(0).$\ \ The interface between the body and the fluid is denoted
by $\partial \mathcal{B}(t)$, the normal vector to the boundary is denoted
by $\boldsymbol{n}(t)$ and it is pointing outside $\mathcal{O}$ and inside $%
\mathcal{B}(t)$. We write
\begin{equation*}
\mathcal{Q}_{\mathcal{F}(t)}=\{(t,\boldsymbol{x})\in \mathbb{R}%
^{1+d}:\,\,t\in \mathbb{R}^{+},\quad \boldsymbol{x}\in \mathcal{F}%
(t)\},\qquad \mathcal{Q}_{\partial \mathcal{B}(t)}=\{(t,\boldsymbol{x})\in
\mathbb{R}^{1+d}:\,\,t\in \mathbb{R}^{+},\quad \boldsymbol{x}\in \partial
\mathcal{B}(t)\}.
\end{equation*}%
The fluid motion is governed by the equations
\begin{equation}
\left\{
\begin{array}{ccc}
\partial _{t}\boldsymbol{u}_{\mathcal{F}}\,+\,\mathrm{div}\,\mathbb{T}(%
\boldsymbol{u}_{\mathcal{F}},p_{\mathcal{F}})\,+\,(\boldsymbol{u}_{\mathcal{F%
}}\cdot \nabla )\boldsymbol{u}_{\mathcal{F}}=\boldsymbol{f}_{0},\qquad
\mathrm{div}\,\boldsymbol{u}_{\mathcal{F}}=0 & \mathrm{in} & \mathcal{Q}_{%
\mathcal{F}(t)}, \\
\boldsymbol{u}_{\mathcal{F}}=\boldsymbol{0} & \mathrm{on} & \partial
\mathcal{O}\times \mathbb{R}^{+} , \\
(\boldsymbol{u}_{\mathcal{F}}-\boldsymbol{u}_{\mathcal{B}})\cdot \boldsymbol{%
n}=0,\qquad 2\mu \lbrack \mathbb{D}(\boldsymbol{u}_{\mathcal{F}})\boldsymbol{%
n}]\times \boldsymbol{n}=-\beta (\boldsymbol{u}_{\mathcal{F}}-\boldsymbol{u}%
_{\mathcal{B}})\times \boldsymbol{n} & \mathrm{on} & \mathcal{Q}_{\partial
\mathcal{B}(t)}, \\
\boldsymbol{u}_{\mathcal{F}}(0)=\boldsymbol{u}_{0} & \mathrm{in} & \mathcal{F%
}_{0},%
\end{array}%
\right.  \label{fluidmotion}
\end{equation}%
where $\boldsymbol{u}_{\mathcal{F}}$ and $p_{\mathcal{F}}$ denote the
velocity and the pressure of the fluid and $\boldsymbol{u}_{\mathcal{B}}$ is
the full velocity of the rigid body. We recall that the rate of the strain
tensor of the fluid and its stress tensor are defined by
\begin{equation*}
\mathbb{D}(\boldsymbol{u}_{\mathcal{F}})=\frac{1}{2}(\nabla \boldsymbol{u}_{%
\mathcal{F}}\,+\,(\nabla \boldsymbol{u}_{\mathcal{F}})^{T})\qquad \text{and}%
\qquad \mathbb{T}(\boldsymbol{u}_{\mathcal{F}},p_{\mathcal{F}})=2\mu \,%
\mathbb{D}(\boldsymbol{u}_{\mathcal{F}})-p_{\mathcal{F}}\mathbb{I},
\end{equation*}%
with $\mu >0$ being the viscosity of the fluid, and $\beta> 0$ is the slip
length.

The fluid equations are coupled to the following balance equations for the
translation velocity $\boldsymbol{\eta }$ and the angular velocity $%
\boldsymbol{\omega }$ of the body,
\begin{equation}
\left\{
\begin{array}{ccc}
m\,\boldsymbol{\eta }^{\prime }(t)+\int_{\partial \mathcal{B}(t)}\mathbb{T}(%
\boldsymbol{u}_{\mathcal{F}},p_{\mathcal{F}})(t,\boldsymbol{x})\boldsymbol{n}%
(t,\boldsymbol{x})\,\mathrm{d}\bm{\sigma }=\boldsymbol{f}_{1}(t), &  &  \\
(\boldsymbol{J}\,\boldsymbol{\omega })^{\prime }(t)\,+\,\int_{\partial
\mathcal{B}(t)}(\boldsymbol{x}-\boldsymbol{x}_{c}(t))\times \mathbb{T}(%
\boldsymbol{u}_{\mathcal{F}},p_{\mathcal{F}})(t,\boldsymbol{x})\boldsymbol{n}%
(t,\boldsymbol{x})\,\mathrm{d}\bm{\sigma }=\boldsymbol{f}_{2}(t) & \mathrm{%
for} & t\in \mathbb{R}^{+}, \\
\boldsymbol{\eta }(0)=\boldsymbol{\eta }_{0},\qquad \boldsymbol{\omega }(0)=%
\boldsymbol{\omega }_{0}, &  &
\end{array}%
\right.  \label{bodymotion}
\end{equation}%
where $m=\rho _{\mathcal{B}}|\mathcal{B}_{0}|$ and $\rho _{\mathcal{B}}$\
are the mass and the constant density of the body, \ $\boldsymbol{x}_{c}$ is
the position of its center of gravity,%
\begin{equation*}
{\boldsymbol{J}}=\rho _{\mathcal{B}}\int_{\mathcal{B}(t)}(|\mathbf{x}-%
\boldsymbol{x}_{c}(t)|^{2}\mathbb{I}-(\mathbf{x}-\boldsymbol{x}%
_{c}(t))\otimes (\mathbf{x}-\boldsymbol{x}_{c}(t)))\,d\mathbf{x}
\end{equation*}%
is the matrix of the inertia moments of the body $\mathcal{B}(t).$ \ The
full velocity of the rigid body is given by
\begin{equation*}
\boldsymbol{u}_{\mathcal{B}}(t,\boldsymbol{x})=\boldsymbol{\eta }(t,%
\boldsymbol{x})+\boldsymbol{\omega }(t)\times (\boldsymbol{x}-\boldsymbol{x}%
_{c}(t)).
\end{equation*}%
The functions $\boldsymbol{f}_{0}$ \ and $\boldsymbol{f}_{1},\ \boldsymbol{f}%
_{2}$ denote the external force and the torques, respectively.

Let us mention that the problem of the motion of one or several rigid bodies
in a viscous fluid filling a bounded domain was investigated by several
authors \cite{CST, DEES1, DEES2, HOST}. In all articles mentioned a non-slip
boundary condition has been considered on the boundaries of the bodies and
of the domain. Hesla \cite{HES} and Hillairet \cite{HIL} have shown that
this condition gives a very paradoxical result of no collisions between the
bodies and the boundary of the domain.

Our article is devoted to the problem of the motion of the rigid body in the
viscous fluid when a slippage is allowed at the fluid-body interface $%
\partial \mathcal{B}(t)$ and a Dirichlet boundary condition on $\partial
\mathcal{O}$. The slippage is prescribed by the Navier boundary condition,
having only the continuity of velocity just in the normal component. We
stress that taking into account slip boundary condition at the interface is
very natural within this model, since the classical Dirichlet boundary
condition leads to unrealistic collision behaviour between the solid and the
domain boundary. Nevertheless, due to the slip condition, the velocity field
is discontinuous across the fluid-solid interface. This makes many aspects
of the theory of weak solutions for Dirichlet conditions inappropriate. It
is worth noting that the case of bounded fluid domain $\mathcal{O}$
furnishes additional difficulty of possible contacts of body and wall. For
this reason, the body needs to start at some distance from the boundary.
Furthermore the lifespan of the solution has to be restricted to a time
interval in which no contacts occur.

To our knowledge the first solvability result was obtained by Neustupa and
Penel \cite{NP1}, \cite{NP2} \ in a particular situation, where they
considered a prescribed collision of a ball with a wall, when the slippage
was allowed on both boundaries. Their pioneer result shows that the slip
boundary condition cleans the no-collision paradox. Recently Gérard--Varet,
Hillairet \cite{GH2} have proved a local-in-time existence result (up to
collisions). The authors of \cite{GHC} have investigated the free fall of a
sphere above a wall, that is when the boundaries are $C^{\infty }$-smooth,
in a viscous incompressible fluid in two different situations: \textit{Mixed
case:} the Navier boundary condition is prescribed on the boundary of the
body and the non-slip boundary condition on the boundary of the domain;
\textit{Slip case:} the Navier boundary conditions are prescribed on both
boundaries, i.e. of the body and of the domain. The result of them is
interesting, saying that in the \textit{Mixed case } \ the sphere never
touches the wall and in the \textit{Slip case } the sphere reaches the wall
during a finite time period.

Recently, the global existence result for a weak solution was proven in the
mixed case, see \cite{CN}, even if the collisions of the body with the
boundary of domain occur in a finite time under a lower regularity of the
body and domain than \cite{GHC}. Our article deals with the strong solution
of the Mixed case. The existence of strong solution was studied by
Takahashi, and Tucsnak \cite{T, TT} in the no-slip boundary conditions and
in the Slip case by Wang \cite{Wa} {\ in the 2D case}.

The plan of the paper is as follows. In Section \ref{Prel} we introduce the
local transformation as in Inoue and Wakimoto \cite{IW}, we define the
functional framework at the basis of our work, we recall also the main
result of this work. Next in Section 3 we prove the existence of solution to
the linearised problem, we consider the non linear problem and we prove the
existence of solution using a fixed point argument.

\section{Preliminaries}

\label{Prel}

\subsection{Local transformation}

\label{Sec:LT} {Since the domain depends on the motion of the rigid body, we
transform the problem to a fixed domain.} There are at least two
possibilities for this transform: the global transformation (cf. \cite{G2,S}%
) is linear, meaning that the whole space is rigidly rotated and shifted
back to its original position at every time $t>0$. A fundamental difficulty
of this approach is that the transformed problem in case of the exterior
domain brings additional terms which are not local perturbation to parabolic
equations and completely change the character of equations. The second one
(cf. \cite{IW}) is characterized by a non-linear local change of coordinates
which only acts in a suitable bounded neighbourhood of the obstacle. The
advantage of the later transform is that it preserves the solenoidal
condition on the fluid velocity, doesn't change the regularity of the
solutions. However the rigid body equations change to become non-linear. Our
analysis is based on the second approach. We define the local transformation
introduced by Inoue and Wakimoto \cite{IW}.

Let $\delta (t)=\mbox{dist }(\mathcal{B}(t),\partial \mathcal{O})$. We fix $%
\delta _{0},$ such that $\delta (t)>\delta _{0},$ and define a $C^{\infty }-$%
smooth solenoidal velocity field $\mathbf{\Lambda }=\mathbf{\Lambda }(t,%
\boldsymbol{x}),$ defined for $t\in \mathbb{R}^{+},$\ $\ \boldsymbol{x}\in
\mathcal{O},$\ satisfying%
\begin{equation*}
\mathbf{\Lambda }(t,\boldsymbol{x})=\left\{
\begin{array}{cc}
0 & \text{\textrm{in the} }\delta _{0}/4\text{ }\mathrm{neighbourhood}\text{
}\mathrm{of}\text{ }\partial \mathcal{O}, \\
\boldsymbol{\eta }(t)+\boldsymbol{\omega }(t)\times (\boldsymbol{x}-%
\boldsymbol{x}_{c}(t)) & \text{\textrm{in the} }\delta _{0}/4\text{ }\mathrm{%
neighbourhood}\text{ }\mathrm{of}\text{ }\mathcal{B}(t).%
\end{array}%
\right.
\end{equation*}%
Then the flow $\boldsymbol{X}(t):\mathcal{O}\rightarrow \mathcal{O}$ is
defined as the solution of the system
\begin{equation}
\,\left.
\begin{array}{c}
\frac{d}{dt}\boldsymbol{X}(t,\boldsymbol{y})=\mathbf{\Lambda }(t,\boldsymbol{%
X}(t,\boldsymbol{y})),\quad \boldsymbol{X}(0,\boldsymbol{y})=\boldsymbol{y}%
,\qquad \forall \boldsymbol{y}\in \mathcal{O}.%
\end{array}%
\right.  \label{eq:3.1}
\end{equation}%
From the results of Takahashi \cite[Lemma 4.2]{T} it follows that (\ref%
{eq:3.1}) has a unique solution. Moreover, the mapping $\boldsymbol{X}$ is a
$C^{\infty }$ diffeomorphism for $\mathcal{O}$ and itself and a
diffeomorphism from $\mathcal{F}_{0}$ onto $\mathcal{F}(t)$ such that the
derivatives
\begin{equation*}
\frac{\partial ^{i+\alpha _{j}}\boldsymbol{X}(t,\boldsymbol{y})}{\partial
t^{i}\partial y_{j}^{\alpha _{j}}},\qquad i\leq 1,\quad \forall \,\alpha
_{j}\geq 0,\quad j=1,...,d,
\end{equation*}%
exist and are continuous. \ Further, denoting $\boldsymbol{Y}$ as the
inverse of $\boldsymbol{X}$ from \cite[Lemma 4.2]{T} it follows that $%
\boldsymbol{Y}$ has also all continuous derivatives%
\begin{equation*}
\frac{\partial ^{i+\alpha _{j}}\boldsymbol{Y}(t,\boldsymbol{x})}{\partial
t^{i}\partial x_{j}^{\alpha _{j}}},\qquad i\leq 1,\quad \forall \,\alpha
_{j}\geq 0,\quad j=1,...,d.
\end{equation*}

Now we introduce the new unknown functions, defined for $\ t\in \mathbb{R}%
^{+}$ \ and $\ \boldsymbol{y}\in \mathcal{O},$%
\begin{eqnarray*}
\widetilde{\boldsymbol{u}}_{\mathcal{F}}(t,\boldsymbol{y}) &=&\mathcal{J}_{%
\boldsymbol{Y}}(t,\boldsymbol{X}(t,\boldsymbol{y}))\boldsymbol{u}_{\mathcal{F%
}}(t,\boldsymbol{X}(t,\boldsymbol{y})),\qquad \widetilde{p}_{\mathcal{F}}(t,%
\boldsymbol{y})=p_{\mathcal{F}}(t,\boldsymbol{X}(t,\boldsymbol{y})), \\
\mathcal{\boldsymbol{T}}(\widetilde{\boldsymbol{u}}_{\mathcal{F}}(t,%
\boldsymbol{y}), \widetilde{p}_{\mathcal{F}}(t,\boldsymbol{y})) &=&%
\boldsymbol{Q}^{T}(t)\mathbb{T}\big(\boldsymbol{Q}(t)\widetilde{\boldsymbol{u%
}}_{\mathcal{F}}(t,\boldsymbol{y}),\widetilde{p}_{\mathcal{F}}(t,\boldsymbol{%
y})\big)\boldsymbol{Q}(t), \\
\widetilde{\boldsymbol{f}_{0}}(t,\boldsymbol{y}) &=&\mathcal{J}_{\boldsymbol{%
Y}}(t,\boldsymbol{X}(t,\boldsymbol{y}))\boldsymbol{f}_{0}(t,\boldsymbol{X}(t,%
\boldsymbol{y})), \\
&& \\
\widetilde{\boldsymbol{\omega }}(t) &=&\boldsymbol{Q}^{T}(t)\boldsymbol{%
\omega }(t),\qquad \widetilde{\boldsymbol{\eta }}(t)=\boldsymbol{Q}^{T}(t)%
\boldsymbol{\eta }(t), \\
\widetilde{\boldsymbol{f}}_{1}(t) &=&\boldsymbol{Q}^{T}(t)\boldsymbol{f}%
_{1}(t),\qquad \widetilde{\boldsymbol{f}}_{2}(t)=\boldsymbol{Q}^{T}(t)%
\boldsymbol{f}_{2}(t),
\end{eqnarray*}%
where $\mathcal{J}_{\boldsymbol{Y}}(t,\boldsymbol{x})=\Big(\frac{\partial
Y_{i}(t,\boldsymbol{x})}{\partial x_{j}}\Big)$\ and $\ \boldsymbol{Q}(t)\in
SO(3)$ is a rotation matrix associated with the rigid body angular velocity $%
\boldsymbol{\omega }$. The transformed normal $\widetilde{\boldsymbol{n}}$
on $\partial \mathcal{B}_{0}$ satisfies $\widetilde{\boldsymbol{n}}=%
\boldsymbol{Q}^{T}(t)\boldsymbol{n}(t)$. The transformed inertia tensor $%
\boldsymbol{I}=\boldsymbol{Q}^{T}(t)\boldsymbol{J}(t)\boldsymbol{Q}(t)$ no
longer depend on time. Furthermore the transformed total force and torque on
the rigid body are given by
\begin{eqnarray*}
\int_{\partial \mathcal{B}(t)}\mathbb{T}(\boldsymbol{u}_{\mathcal{F}},p_{%
\mathcal{F}})\boldsymbol{n}(t)\,\mathrm{d}\bm{\sigma } &=&\boldsymbol{Q}%
\int_{\partial \mathcal{B}_{0}}\boldsymbol{\mathcal{T}}(\widetilde{%
\boldsymbol{u}}_{\mathcal{F}},\widetilde{p}_{\mathcal{F}})\widetilde{%
\boldsymbol{n}}\,\mathrm{d}\bm{\sigma }(\boldsymbol{y}), \\
\int_{\partial \mathcal{B}(t)}(\boldsymbol{x}-\boldsymbol{x}_{c}(t))\times
\mathbb{T}(\boldsymbol{u}_{\mathcal{F}},p_{\mathcal{F}})\boldsymbol{n}(t)\,%
\mathrm{d}\bm{\sigma } &=&\boldsymbol{Q}\int_{\partial \mathcal{B}_{0}}%
\boldsymbol{y}\times \boldsymbol{\mathcal{T}}(\widetilde{\boldsymbol{u}}_{%
\mathcal{F}},\widetilde{p}_{\mathcal{F}})\widetilde{\boldsymbol{n}}\,\mathrm{%
d}\bm{\sigma }(\boldsymbol{y}).
\end{eqnarray*}%
Thus for some $T>0,$ \ that will be founded later on,\ the new unknowns $%
\widetilde{\boldsymbol{u}}_{\mathcal{F}}$, $\widetilde{p}_{\mathcal{F}}$ and
$\widetilde{\boldsymbol{\eta }},$ $\widetilde{\boldsymbol{\omega }},$\
defined on the cylindrical domains $(0,T)\times \mathcal{F}_{0}$ and $%
(0,T)\times \mathcal{B}_{0},$\ satisfy the following system of equations%
\begin{equation}
\left\{
\begin{array}{ccc}
\partial _{t}\widetilde{\boldsymbol{u}}_{\mathcal{F}}\,+\,(\mathcal{M}-\mu
\mathcal{L})\widetilde{\boldsymbol{u}}_{\mathcal{F}}\,+\,\mathcal{N}%
\widetilde{\boldsymbol{u}}_{\mathcal{F}}+\,\mathcal{G}\widetilde{p}_{%
\mathcal{F}}\,=\,\widetilde{\boldsymbol{f}_{0}},\qquad \mathrm{div}\,%
\widetilde{\boldsymbol{u}}_{\mathcal{F}}\,=\,0 & \mathrm{in} & (0,T)\times
\mathcal{F}_{0}, \\
\widetilde{\boldsymbol{u}}_{\mathcal{F}}\,=\,\mathbf{0}\quad \mathrm{on}%
\quad (0,T)\times \partial \mathcal{O},\qquad \widetilde{\boldsymbol{u}}_{%
\mathcal{F}}(0)=\boldsymbol{u}_{0}\quad \mathrm{in}\quad \mathcal{F}_{0}, &
&  \\
&  &  \\
(\widetilde{\boldsymbol{u}}_{\mathcal{F}}-\widetilde{\boldsymbol{u}}_{%
\mathcal{B}})\cdot \widetilde{\boldsymbol{n}}=\boldsymbol{0},\qquad 2\mu
\lbrack \mathbb{D}(\widetilde{\boldsymbol{u}}_{\mathcal{F}})\widetilde{%
\boldsymbol{n}}]\times \widetilde{\boldsymbol{n}}=-\beta (\widetilde{%
\boldsymbol{u}}_{\mathcal{F}}-\widetilde{\boldsymbol{u}}_{\mathcal{B}%
})\times \widetilde{\boldsymbol{n}} & \mathrm{on} & (0,T)\times \partial
\mathcal{B}_{0}, \\
&  &  \\
m\,\widetilde{\boldsymbol{\eta }}^{\prime }\,\,-m\,(\widetilde{\boldsymbol{%
\omega }}\times \widetilde{\boldsymbol{\eta }})\,+\int_{\partial \mathcal{B}%
_{0}}\boldsymbol{\mathcal{T}}(\widetilde{\boldsymbol{u}}_{\mathcal{F}},%
\widetilde{p}_{\mathcal{F}})\widetilde{\boldsymbol{n}}\,\mathrm{d}\bm{\sigma
}=\widetilde{\boldsymbol{f}}_{1}(t), &  &  \\
\mathbf{I}\widetilde{\boldsymbol{\omega }}^{\prime }\,-\,\widetilde{%
\boldsymbol{\omega }}\times (\mathbf{I}\widetilde{\boldsymbol{\omega }}%
)\,+\,\int_{\partial \mathcal{B}_{0}}\boldsymbol{y}\times \boldsymbol{%
\mathcal{T}}(\widetilde{\boldsymbol{u}}_{\mathcal{F}},\widetilde{p}_{%
\mathcal{F}})\widetilde{\boldsymbol{n}}\,\mathrm{d}\bm{\sigma }=\widetilde{%
\boldsymbol{f}}_{2}(t), & \mathrm{for}\, & t\in (0,T), \\
\widetilde{\boldsymbol{\eta }}(0)=\boldsymbol{\eta }_{0},\qquad \widetilde{%
\boldsymbol{\omega }}(0)=\boldsymbol{\omega }_{0} &  &
\end{array}%
\right.  \label{modifiedsystem}
\end{equation}%
with $\widetilde{\boldsymbol{u}}_{\mathcal{B}}=\widetilde{\boldsymbol{\eta }}%
+\widetilde{\boldsymbol{\omega }}\times \boldsymbol{y}$ and the convection
term is transformed into
\begin{equation*}
(\mathcal{N}\mathbf{u})_{i}=\sum_{j=1}^{d}u_{j}\partial
_{j}u_{i}+\sum_{j,k+1}^{d}\Gamma _{jk}^{i}u_{j}u_{k},\qquad i=1,...,d.
\end{equation*}%
\ The transformed time derivative $\mathcal{M}\boldsymbol{u}$ and the
gradient $\mathcal{G}p$ are calculated by
\begin{equation*}
(\mathcal{M}\mathbf{u})_{i}=\sum_{j=1}^{d}\dot{Y}_{j}\partial
_{j}u_{i}+\sum_{j,k=1}^{d}\left( \Gamma _{jk}^{i}\dot{Y}_{k}+(\partial
_{k}Y_{i})(\partial _{j}\dot{X}_{k})\right) u_{j},\qquad (\mathcal{G}%
p)_{i}=\sum_{j=1}^{d}g^{ij}\partial _{j}p.
\end{equation*}%
Moreover the operator $\mathcal{L}$ denotes the transformed Laplace
operator, having the components
\begin{equation*}
(\mathcal{L}\mathbf{u})_{i}=\sum_{j,k=1}^{d}\partial _{j}(g^{jk}\partial
u_{i})+2\sum_{j,k,l=1}^{d}g^{kl}\Gamma _{jk}^{i}\partial
_{l}u_{j}+\sum_{j,k,l=1}^{d}\left( \partial _{k}(g^{kl}\Gamma
_{kl}^{i})+\sum_{m=1}^{n}g^{kl}\Gamma _{jl}^{m}\Gamma _{km}^{i}\right) u_{j}.
\end{equation*}%
The coefficients are given by the metric covariant tensor \ $%
g_{ij}=X_{k,i}X_{k,j},$ \ \ the metric contra-variant tensor \ $%
g^{ij}=Y_{i,k}Y_{j,k}$ \ \ \ and the Christoffel symbols
\begin{equation*}
\Gamma _{ij}^{k}=\frac{1}{2}g^{kl}(g_{il,j}+g_{jl,i}-g_{ij,l}).
\end{equation*}%
It is easy to observe that in particular it holds $\Gamma
_{ij}^{k}=Y_{k,l}X_{l,ij}.$ \ As described in \cite{IW}, problem %
\eqref{fluidmotion}--\eqref{bodymotion} is equivalent to problem %
\eqref{modifiedsystem} and a solution to the transformed problem %
\eqref{modifiedsystem} yields a solution to the initial problem %
\eqref{fluidmotion}--\eqref{bodymotion}.

\bigskip

\subsection{Function spaces and the main theorem}

In the sequel we use the following function spaces, defined on the moving
domain $(0,T)\times \mathcal{F}(t),$%
\begin{equation*}
L^{2}(0,T;H^{2}(\mathcal{F}(t))),\qquad C([0,T];H^{1}(\mathcal{F}%
(t))),\qquad H^{1}(0,T;L^{2}(\mathcal{F}(t))),\qquad L^{2}(0,T;H^{1}(%
\mathcal{F}(t))).
\end{equation*}%
If we consider $\boldsymbol{U}_{\mathcal{F}}(t,\boldsymbol{y}):\mathcal{F}%
_{0}\rightarrow \mathbb{R}^{d},$ which is calculated as
\begin{equation*}
\boldsymbol{U}_{\mathcal{F}}(t,\boldsymbol{y})=\boldsymbol{u}_{\mathcal{F}%
}(t,\boldsymbol{X}(t,\boldsymbol{y}))\qquad \text{for any function\quad }%
\boldsymbol{u}_{\mathcal{F}}(t,\cdot ):\mathcal{F}(t)\rightarrow \mathbb{R}%
^{d},
\end{equation*}%
then above mentioned function spaces can be redefined in the fixed domain $%
(0,T)\times \mathcal{F}_{0}.$ For instance
\begin{equation*}
L^{2}(0,T;H^{2}(\mathcal{F}_{0}))=\{\boldsymbol{U}_{\mathcal{F}}:\boldsymbol{%
u}_{\mathcal{F}}\in L^{2}(0,T;H^{2}(\mathcal{F}(t)))\}.
\end{equation*}

\bigskip

Now we can formulate the main result.

\begin{teo}
\label{th:2.1} Suppose that $\overline{\mathcal{B}_{0}}\subset \mathcal{O}$\
and
\begin{eqnarray}
\boldsymbol{u}_{0} &\in &H^{1}(\mathcal{F}_{0}),\qquad \boldsymbol{u}_{%
\mathcal{B},0}={\boldsymbol{\eta }_{0}+\boldsymbol{\omega }_{0}\times (%
\boldsymbol{x}-\boldsymbol{x}_{c}(0))}\in H^{1}(\mathcal{B}_{0}),  \notag \\
\boldsymbol{f}_{0} &\in &L_{loc}^{2}(\mathbb{R}^{+};H^{1}(\mathcal{F}%
_{0})),\qquad \boldsymbol{f}_{1},\boldsymbol{f}_{2}\in L_{loc}^{2}(\mathbb{R}%
^{+}),  \label{R}
\end{eqnarray}%
that satisfy
\begin{equation*}
(\boldsymbol{u}_{0}-\boldsymbol{u}_{\mathcal{B},0})\cdot \boldsymbol{n}%
|_{\partial \mathcal{B}_{0}}=0,\qquad \boldsymbol{u}_{0}|_{\partial \mathcal{%
O}}=\mathbf{0},\qquad \mathrm{div}\ \boldsymbol{u}_{0}=0\qquad \mathrm{in}%
\quad \mathcal{F}_{0}.
\end{equation*}%
Then there exists $T_{0}>0$ such that \eqref{fluidmotion}-\eqref{bodymotion}
has a unique solution which satisfies for all $T<T_{0}$
\begin{equation*}
\boldsymbol{u}_{\mathcal{F}},\,\ p_{\mathcal{F}},\,\ \boldsymbol{\eta }(t),%
\boldsymbol{\omega }(t)\in \mathcal{U}_{T}(\mathcal{F}(t))\times
L^{2}(0,T;H^{1}(\mathcal{F}(t)))\times H^{1}(0,T)\times H^{1}(0,T),
\end{equation*}%
where
\begin{equation*}
\mathcal{U}_{T}(\mathcal{F}(t))=L^{2}(0,T;H^{2}(\mathcal{F}(t)))\cap
C(0,T;H^{1}(\mathcal{F}(t)))\cap H^{1}(0,T;L^{2}(\mathcal{F}(t))).
\end{equation*}
\end{teo}

\bigskip

\bigskip

\section{Strong solution}

\label{strongsolusect}

\subsection{Stokes problem}

We will consider the following linearized system, which couples Stokes type
equations and linear ordinary differential equations,%
\begin{equation}
\left\{
\begin{array}{ccc}
\partial _{t}\,\mathbf{z}_{\mathcal{F}}-\mu \,\Delta \mathbf{z}_{\mathcal{F}%
}\,+\,\nabla q_{\mathcal{F}}\,=\,\boldsymbol{F}_{0},\qquad \mathrm{div}\,%
\mathbf{z}_{\mathcal{F}}\,=\,0 & \mathrm{in} & (0,T)\times \mathcal{F}_{0},
\\
\mathbf{z}_{\mathcal{F}}\,=\,\mathbf{0}\quad \mathrm{on}\quad (0,T)\times
\partial \mathcal{O},\qquad \mathbf{z}_{\mathcal{F}}(0)=\boldsymbol{u}%
_{0}\quad \mathrm{in}\quad \mathcal{F}_{0}, &  &  \\
&  &  \\
(\mathbf{z}_{\mathcal{F}}-\mathbf{z}_{\mathcal{B}})\cdot \widetilde{%
\boldsymbol{n}}\,=\,0,\qquad 2\mu \lbrack \mathbb{D}(\mathbf{z}_{\mathcal{F}%
})\widetilde{\boldsymbol{n}}]\times \widetilde{\boldsymbol{n}}=-\beta (%
\mathbf{z}_{\mathcal{F}}-\mathbf{z}_{\mathcal{B}})\times \widetilde{%
\boldsymbol{n}}\, & \mathrm{on} & (0,T)\times \partial \mathcal{B}_{0}, \\
&  &  \\
m\,\boldsymbol{\xi }^{\prime }\,+\,\int_{\partial \mathcal{B}_{0}}\mathbb{T}(%
\mathbf{z}_{\mathcal{F}},q_{\mathcal{F}})\widetilde{\boldsymbol{n}}\,\mathrm{%
d}\bm{\sigma }\,=\mathbf{F}_{1}, &  &  \\
{\boldsymbol{I}}\,\mathbf{w}^{\prime }\,+\,\int_{\partial \mathcal{B}_{0}}%
\boldsymbol{y}\times \mathbb{T}(\mathbf{z}_{\mathcal{F}},q_{\mathcal{F}})%
\widetilde{\boldsymbol{n}}\,\mathrm{d}\bm{\sigma }\,=\mathbf{F}_{2}, &
\mathrm{for} & t\in (0,T), \\
\boldsymbol{\xi }(0)=\boldsymbol{\eta }_{0},\qquad \mathbf{w}(0)=\boldsymbol{%
\omega }_{0} &  &
\end{array}%
\right.  \label{lin}
\end{equation}%
with $\mathbf{z}_{\mathcal{B}}=\boldsymbol{\xi }+\mathbf{w}\times
\boldsymbol{y}.$

Let us recall a well-known result (see Kato \cite{K1,K2}).

\begin{proposition}
\label{abstractresult} Let $H$ be a Hilbert space. Let $\mathbb{A}:D(\mathbb{%
A})\rightarrow H$ be a self adjoint and accretive operator. If $\
\boldsymbol{F}\in L^{2}(0,T;H)$, $\boldsymbol{u}_{0}\in D(\mathbb{A}^{1/2})$%
, then the problem
\begin{equation*}
\boldsymbol{u}^{\prime }+\mathbb{A}\boldsymbol{u}=\boldsymbol{F},\qquad
\boldsymbol{u}(0)=\boldsymbol{u}_{0},
\end{equation*}%
has a unique solution $\boldsymbol{u}\in L^{2}(0,T;D(\mathbb{A}))\cap
C([0,T];D(\mathbb{A}^{1/2}))\cap H^{1}(0,T;H),$ which satisfies
\begin{equation*}
\Vert \boldsymbol{u}\Vert _{L^{2}(0,T;D(\mathbb{A}))}+\Vert \boldsymbol{u}%
\Vert _{C([0,T];D(\mathbb{A}^{1/2}))}+\Vert \boldsymbol{u}\Vert
_{H^{1}(0,T;H)}\leq C(\Vert \boldsymbol{u}_{0}\Vert _{D(\mathbb{A}%
^{1/2})}+\Vert \boldsymbol{F}\Vert _{L^{2}(0,T;H)})
\end{equation*}%
with a constant $C$ depending on the operator $\mathbb{A}$ and the time $T$.
Moreover, the constant $C$ is a non decreasing function of $T$.
\end{proposition}

Let us define the functional spaces
\begin{align*}
\mathcal{H}&=\{ \boldsymbol{\phi }\in L^{2}(\mathcal{O}):\; \mathrm{div}\,%
\boldsymbol{\phi }=0 \quad \text{in \ }\mathcal{O},\quad \text{such that}%
\quad \boldsymbol{\phi }|_{\mathcal{F}_{0}}=\boldsymbol{\phi }_{\mathcal{F}%
}\in \mathcal{D}^{\prime }(\mathcal{F}_{0}),\quad \boldsymbol{\phi }|_{%
\mathcal{B}_{0}}=\boldsymbol{\phi }_{\mathcal{B}}\in \mathcal{R}\}, \\
\mathcal{V}& =\{\boldsymbol{\phi }\in \mathcal{H}:\quad \boldsymbol{\phi }_{%
\mathcal{F}}\in H^{1}(\mathcal{F}_{0}),\qquad \boldsymbol{\phi }_{\mathcal{F}%
}|_{\partial \mathcal{O}}=0,\qquad (\boldsymbol{\phi }_{\mathcal{F}}-%
\boldsymbol{\phi }_{\mathcal{B}}){\cdot \widetilde{\boldsymbol{n}}}_{\mid
\partial \mathcal{B}_{0}}=0\},
\end{align*}%
where
\begin{equation*}
\mathcal{R}=\{\boldsymbol{\phi }:\quad \boldsymbol{\phi }(\boldsymbol{y})=%
\boldsymbol{\xi }_{\phi }+\mathbf{w}_{\phi }\times \boldsymbol{y}\qquad
\text{with}\quad \boldsymbol{\xi }_{\phi },\mathbf{w}_{\phi }\in \mathbb{R}%
^{d}\}.
\end{equation*}%
\

For $\boldsymbol{u},\boldsymbol{v}\in \mathcal{H}$ we define the inner
product
\begin{equation*}
(\boldsymbol{u},\boldsymbol{v})=\int_{\mathcal{F}_{0}}\boldsymbol{u}_{%
\mathcal{F}}\cdot \boldsymbol{v}_{\mathcal{F}}\,\mathrm{d}\boldsymbol{y}%
+\int_{\mathcal{B}_{0}}\rho _{\mathcal{B}}\boldsymbol{u}_{\mathcal{B}}\cdot
\boldsymbol{v}_{\mathcal{B}}\,\mathrm{d}\boldsymbol{y},
\end{equation*}%
which equals to
\begin{equation}
(\boldsymbol{u},\boldsymbol{v})=\int_{\mathcal{F}_{0}}\boldsymbol{u}_{%
\mathcal{F}}\cdot \boldsymbol{v}_{\mathcal{F}}\,\mathrm{d}\boldsymbol{y}+m%
\boldsymbol{\xi }_{\boldsymbol{u}_{\mathcal{B}}}\cdot \boldsymbol{\xi }_{%
\boldsymbol{v}_{\mathcal{B}}}+(\boldsymbol{I}\mathbf{w}_{\boldsymbol{u}_{%
\mathcal{B}}})\cdot \mathbf{w}_{\mathbf{v}_{\mathcal{B}}}.  \label{inner}
\end{equation}

Let us denote%
\begin{equation*}
{\mathcal{A}}\mathbf{z}(\boldsymbol{y})=\left\{
\begin{array}{l}
-\mu \Delta \mathbf{z}_{\mathcal{F}}(\boldsymbol{y}),\qquad \boldsymbol{y}%
\in \mathcal{F}_{0}, \\
{\frac{2\mu }{m}\int_{\partial \mathcal{B}_{0}}\mathbb{D}(\mathbf{z}_{%
\mathcal{F}})\mathbf{\widetilde{\boldsymbol{n}}}}\,\,\mathrm{d}\bm{\sigma }+%
\Big(2\mu \boldsymbol{I}^{-1}\int_{\partial \mathcal{B}_{0}}\mathbb{D}(%
\mathbf{z}_{\mathcal{F}})\widetilde{\boldsymbol{n}}\times \boldsymbol{y}\,%
\mathrm{d}\bm{\sigma }\Big)\times \boldsymbol{y},\qquad \boldsymbol{y}\in
\mathcal{B}_{0},%
\end{array}%
\right.
\end{equation*}%
and define the operator
\begin{equation}
{A}\mathbf{z}={\mathbb{P}}{\mathcal{A}}\mathbf{z}\qquad \text{for any }\,%
\boldsymbol{z}\in D(A),  \label{operA}
\end{equation}%
where ${\mathbb{P}}:L^{2}(\mathcal{O})\rightarrow \mathcal{H}$\ \ is the
orthogonal projector on $\mathcal{H}$ in $L^{2}(\mathcal{O})$ and the domain
of the operator of $A$ is defined by
\begin{multline}
D(A)=\{\boldsymbol{\phi }\in \mathcal{H}:\quad \boldsymbol{\phi }_{\mathcal{F%
}}\in H^{2}(\mathcal{F}_{0}),\qquad \,\boldsymbol{\phi }_{\mathcal{F}%
}|_{\partial \mathcal{O}}=0,  \notag \\
(\boldsymbol{\phi }_{\mathcal{F}}-\boldsymbol{\phi }_{\mathcal{B}})\cdot
\widetilde{\boldsymbol{n}}\mid _{\partial \mathcal{B}_{0}}=0,\qquad 2\mu (%
\mathbb{D}(\boldsymbol{\phi }_{\mathcal{F}})\cdot \widetilde{\boldsymbol{n}}%
)\times \widetilde{\boldsymbol{n}}\mid _{\partial \mathcal{B}_{0}}=-\beta (%
\boldsymbol{\phi }_{\mathcal{F}}-\boldsymbol{\phi }_{\mathcal{B}})\times
\widetilde{\boldsymbol{n}}\mid _{\partial \mathcal{B}_{0}} \},
\label{operA2}
\end{multline}%
%
%
%

\begin{proposition}
\label{analyticityA} The operator ${\ A}$ defined by (\ref{operA}) is self
adjoint and positive. Consequently $A$ is a generator of contraction
analytic semi-group in $\mathcal{H}$. Moreover, there exists a constant $%
C>0, $ such that for any $\mathbf{z}\in D({A})$ we have
\begin{equation*}
\Vert \mathbf{z}_{\mathcal{F}}\Vert _{H^{2}(\mathcal{F}_{0})}\,+\,\Vert
\mathbf{z}_{\mathcal{B}}\Vert _{H^{2}(\mathcal{B}_{0})} \leq C\Vert(\mathbb{I%
}+ {A})\mathbf{z}\Vert _{L^{2}(\mathcal{O})}.
\end{equation*}
\end{proposition}

\begin{proof}
\textbf{(i) $A$ is symmetric.} Let $\mathbf{z},\mathbf{v}\in D(A)$. Then the
integration by parts used twicely gives that
\begin{eqnarray*}
(A\mathbf{z},\mathbf{v}) &=&2\mu \int_{\mathcal{F}_{0}}\mathbb{D}(\mathbf{z}%
_{\mathcal{F}}):\mathbb{D}(\mathbf{v}_{\mathcal{F}})\,\mathrm{d}\boldsymbol{y%
}+\beta \int_{\partial \mathcal{B}_{0}}[(\mathbf{z}_{\mathcal{F}}-\mathbf{z}%
_{\mathcal{B}})\times \widetilde{\boldsymbol{n}}]\cdot \lbrack (\mathbf{v}_{%
\mathcal{F}}-\mathbf{v}_{\mathcal{B}})\times \widetilde{\boldsymbol{n}}]\,%
\mathrm{d}\bm{\sigma } \\
&=&(\mathbf{z},A\mathbf{v}).
\end{eqnarray*}
Hence $A$ is a symmetric operator.

\noindent \textbf{(ii) $A$ is positive.} From \textbf{(i)} we have that%
\begin{equation*}
(A\mathbf{z},\mathbf{z})\,=2\mu \Vert \mathbb{D}(\mathbf{z}_{\mathcal{F}%
})\Vert _{L^{2}(\mathcal{F}_{0})}^{2}+\beta \int_{\partial \mathcal{B}_{0}}|%
\mathbf{z}_{\mathcal{F}}-\mathbf{z}_{\mathcal{B}}|^{2}\,\mathrm{d}\bm{\sigma
}\qquad \text{for any }\,\boldsymbol{z}\in D(A).
\end{equation*}%
Thus $A$ is a positive operator.

\noindent \textbf{(iii) $A$ is self-adjoint.} In order to prove that $A$ is
self adjoint, it suffices to prove that the operator $\mathbb{I}%
+A:D(A)\rightarrow \mathcal{H}$ is surjective.

First, let us note that the solution $\mathbf{z}\in D(A)$ of the problem $(%
\mathbb{I}+A)\mathbf{z}=\boldsymbol{F}\in \mathcal{H}$ \ in the weak
formulation\ satisfies the integral equality%
\begin{equation*}
(\mathbf{z},\mathbf{v})+(A\mathbf{z},\mathbf{v})=(\boldsymbol{F},\mathbf{v}%
)\qquad \text{for any}\quad \mathbf{v}\in \mathcal{V},
\end{equation*}%
that is
\begin{align*}
(\mathbf{z},\mathbf{v})\,& +\,2\mu \int_{\mathcal{F}_{0}}\mathbb{D}(\mathbf{z%
}_{\mathcal{F}}):\mathbb{D}(\mathbf{v}_{\mathcal{F}})\,\mathrm{d}\boldsymbol{%
y}\,\, \\
& +\beta \int_{\partial \mathcal{B}_{0}}(\mathbf{z}_{\mathcal{F}}-\mathbf{z}%
_{\mathcal{B}})\cdot (\mathbf{v}_{\mathcal{F}}-\mathbf{v}_{\mathcal{B}})\,%
\mathrm{d}\bm{\sigma }=(\boldsymbol{F},\mathbf{v})\qquad \text{for any}\quad
\mathbf{v}\in \mathcal{V}
\end{align*}

Let us define the bilinear form $a:\mathcal{V}\times \mathcal{V}\rightarrow
\mathbb{R}$ \ \ by
\begin{align}
a(\mathbf{z},\mathbf{v})=(\mathbf{z},\mathbf{v})\,& +\,2\mu \int_{\mathcal{F}%
_{0}}\mathbb{D}(\mathbf{z}_{\mathcal{F}}):\mathbb{D}(\mathbf{v}_{\mathcal{F}%
})\,\mathrm{d}\boldsymbol{y}\,  \notag \\
& +\beta \int_{\partial \mathcal{B}_{0}}(\mathbf{z}_{\mathcal{F}}-\mathbf{z}%
_{\mathcal{B}})\cdot (\mathbf{v}_{\mathcal{F}}-\mathbf{v}_{\mathcal{B}})\,%
\mathrm{d}\bm{\sigma }\qquad \text{for any}\quad \mathbf{z},\mathbf{v}\in
\mathcal{V}.  \label{fluid-solid-bilinear}
\end{align}%
Using the positivity of the operator $A,$ we easily check that $a$ is a
bilinear continuous coercive form on $\mathcal{V}$. \ Furthermore the
mapping $\mathbf{v}\rightarrow (\boldsymbol{F},\mathbf{v})$ is a continuous
linear form on $\mathcal{V}$. \ Therefore the Lax-Milgram theorem implies
the existence of a unique solution $\mathbf{z}\in \mathcal{V}$ of the
problem \eqref{fluid-solid-bilinear}. Using \cite{ShibataShimada} we deduce
that there exists $q_{\mathcal{F}}\in \mathcal{D}^{\prime }(\mathcal{F}%
_{0}), $ such that
\begin{equation*}
\mathbf{z}_{\mathcal{F}}\,-\,\mu \Delta \mathbf{z}_{\mathcal{F}}\,+\,\nabla
\,q_{\mathcal{F}}\,=\,\boldsymbol{F}\qquad \mathrm{in}\quad \mathcal{D}%
^{\prime }(\mathcal{F}_{0}).
\end{equation*}%
In addition, $\mathbf{z}_{\mathcal{F}}$ is a unique weak solution of the
system
\begin{equation*}
\left\{
\begin{array}{ccc}
\mathbf{z}_{\mathcal{F}}\,-\mu \,\Delta \mathbf{z}_{\mathcal{F}}+\nabla q_{%
\mathcal{F}}=\boldsymbol{F},\qquad \mathrm{div}\,\mathbf{z}_{\mathcal{F}}=0
& \mathrm{in} & \mathcal{F}_{0}, \\
(\mathbf{z}_{\mathcal{F}}-\mathbf{z}_{\mathcal{B}})\cdot \widetilde{%
\boldsymbol{n}}\,=\,0,\qquad 2\mu \lbrack \mathbb{D}(\mathbf{z}_{\mathcal{F}%
})\widetilde{\boldsymbol{n}}]\times \widetilde{\boldsymbol{n}}=-\beta (%
\mathbf{z}_{\mathcal{F}}-\mathbf{z}_{\mathcal{B}})\times \widetilde{%
\boldsymbol{n}} & \mathrm{on} & \partial \mathcal{B}_{0}, \\
\mathbf{z}_{\mathcal{F}}\,=\,\mathbf{0}\quad \mathrm{on}\quad \partial
\mathcal{O} &  &
\end{array}%
\right.
\end{equation*}%
and it satisfies \ the estimate%
\begin{equation*}
\Vert \mathbf{z}_{\mathcal{F}}\Vert _{H^{2}(\mathcal{F}_{0})}\,\leq
\,C\,(\Vert \boldsymbol{F}\Vert _{L^{2}(\mathcal{F}_{0})}\,+\,\Vert \mathbf{z%
}_{\mathcal{B}}\Vert _{H^{3/2}(\partial \mathcal{B}_{0})}).
\end{equation*}%
On the other hand, since $\mathbf{z}_{\mathcal{B}}\in \mathcal{R},$ there
exist two vectors $\boldsymbol{\xi },\,\mathbf{w}\in \mathbb{R}^{d},$ such
that $\mathbf{z}_{\mathcal{B}}=\boldsymbol{\xi }\,+\,\mathbf{w}\times
\boldsymbol{y}$\textrm{\ in }$\mathcal{B}_{0},$ that gives
\begin{equation*}
\Vert \mathbf{z}_{\mathcal{B}}\Vert _{H^{2}(\mathcal{B}_{0})}\,\leq
\,C\,\Vert \boldsymbol{F}\Vert _{L^{2}(\mathcal{O})}.
\end{equation*}%
Hence we conclude that
\begin{equation*}
\Vert \mathbf{z}_{\mathcal{F}}\Vert _{H^{2}(\mathcal{F}_{0})}\,+\,\Vert
\mathbf{z}_{\mathcal{B}}\Vert _{H^{2}(\mathcal{B}_{0})}\,\leq \,C\,\Vert (%
\mathbb{I}+A)\mathbf{z}\Vert _{L^{2}(\mathcal{O})}.
\end{equation*}%
$\square $
\end{proof}

Now we are in a position to prove the following result for the linearised
fluid-structure problem (\ref{lin}).

\begin{proposition}
Let $T>0$. If
\begin{equation*}
\widetilde{\boldsymbol{u}}_{0}=(\widetilde{\boldsymbol{u}}_{\mathcal{F},0},%
\widetilde{\boldsymbol{u}}_{\mathcal{B},0})\in \mathcal{V},\qquad
\boldsymbol{F}_{0}\in L^{2}(0,T;L^{2}(\mathcal{F}_{0}))\quad \text{and}\quad
\boldsymbol{F}_{1},\boldsymbol{F}_{2}\in L^{2}(0,T),
\end{equation*}%
then problem (\ref{lin}) has a unique solution on $[0,T]$, that satisfies a
priori estimate
\begin{align}
\Vert \mathbf{z}_{\mathcal{F}}\Vert _{\mathcal{U}_{T}(\mathcal{F}_{0})}&
+\Vert \nabla q_{\mathcal{F}}\Vert _{L^{2}(0,T;L^{2}(\mathcal{O}))}+\Vert
\boldsymbol{\xi }\Vert _{H^{1}(0,T)}+\Vert \mathbf{w}\Vert _{H^{1}(0,T)}
\notag \\
& \leq C(\Vert (\boldsymbol{F}_{1},\boldsymbol{F}_{2})\Vert
_{L^{2}(0,T)}+\Vert \boldsymbol{F}_{0}\Vert _{L^{2}(0,T;L^{2}(\mathcal{F}%
_{0}))}+\Vert \widetilde{\boldsymbol{u}}_{\mathcal{B},0}\Vert _{H^{1}(%
\mathcal{B}_{0})}+\Vert \widetilde{\boldsymbol{u}}_{0}\Vert _{H^{1}(\mathcal{%
F}_{0})}),  \label{est}
\end{align}%
with $C$ is a nondecreasing function of $T$.
\end{proposition}

\begin{proof}
We follow Wang verbatim \cite{Wa}. The difference between Wang´s problem and
our problem is that, Wang considered slip boundary conditions on both
boundaries and we consider the Mixed case. Moreover, in \cite{Wa} only 2D
case is investigated. We consider 3D case. For completeness, we will give
the principal part of the proof.

We will show that the linearized fluid-solid problem \eqref{lin} can be
written in the form
\begin{equation}
\partial _{t}\mathbf{z}+A\mathbf{z}=\boldsymbol{F},\qquad \mathbf{z}(0)=%
\widetilde{\boldsymbol{u}}_{0},  \label{newform}
\end{equation}%
where%
\begin{equation*}
\mathbf{z}=\mathbf{z}_{\mathcal{F}}1_{\mathcal{F}_{0}}+\mathbf{z}_{\mathcal{B%
}}1_{\mathcal{B}_{0}},\qquad \widetilde{\boldsymbol{u}}_{0}=\mathbf{z}_{%
\mathcal{F}}(0)1_{\mathcal{F}_{0}}+\mathbf{z}_{\mathcal{B}}(0)1_{\mathcal{B}%
_{0}}
\end{equation*}%
and
\begin{equation*}
\boldsymbol{F}=\mathbb{P}\left( \mathbf{F}_{0}1_{\mathcal{F}_{0}}+\Big(\frac{%
\boldsymbol{F}_{1}}{m}+\boldsymbol{I}^{-1}\boldsymbol{F}_{2}\times
\boldsymbol{y}\Big)1_{\mathcal{B}_{0}}\right) .
\end{equation*}%
By Proposition \ref{analyticityA}, the fluid-solid operator $%
A:D(A)\rightarrow \mathcal{H}$ \ is a positive self adjoint operator. Thus
by Proposition \ref{abstractresult}, the problem \eqref{lin} has a unique
solution
\begin{equation*}
\mathbf{z}\in L^{2}(0,T;\,D(A))\cap C([0,T];\,D(A^{1/2}))\cap H^{1}(0,T;%
\mathcal{H}).
\end{equation*}%
Recall that the norm of $D(A^{1/2})$ is equivalent to the norm of $\mathcal{V%
}$.

Since $\mathbf{z}\in H^{1}(0,T;\mathcal{H}),$ there exist two vector
functions $\boldsymbol{\xi },\mathbf{w}\in H^{1}(0,T),$ such that
\begin{equation*}
\mathbf{z}_{\mathcal{B}}(t,\boldsymbol{y})=\boldsymbol{\xi }(t)+\mathbf{w}%
(t)\times \boldsymbol{y}\qquad \text{for any }\quad \boldsymbol{y}\in
\mathcal{B}_{0}.
\end{equation*}%
If we take the inner product \eqref{inner} of equality \eqref{newform}$_{1}$
and $\boldsymbol{\phi }\in \mathcal{H}$, \ we get
\begin{align}
\int_{\mathcal{F}_{0}}&\mathbf{z}_{\mathcal{F}}^{\prime }\cdot \boldsymbol{%
\phi }_{\mathcal{F}}\,\,\mathrm{d}\boldsymbol{y}\,+\,m\Big(\boldsymbol{\xi }%
^{\prime }\,-\,\frac{\boldsymbol{F}_{1}}{m}\Big)\cdot \boldsymbol{\xi }_{%
\boldsymbol{\phi }}\,+\,\boldsymbol{I}\Big(\mathbf{w}^{\prime }\,-\,%
\boldsymbol{I}^{-1}\boldsymbol{F}_{2}\Big)\cdot \mathbf{w}_{\phi }\,-\,\int_{%
\mathcal{F}_{0}}\mu \Delta \mathbf{z}_{\mathcal{F}}\cdot \boldsymbol{\phi }_{%
\mathcal{F}}\,\,\mathrm{d}\boldsymbol{y}  \notag \\
&+\,2\,\mu \,\int_{\partial \mathcal{B}_{0}}\mathbb{D}(\mathbf{z}_{\mathcal{F%
}})\widetilde{\boldsymbol{n}}\,\mathrm{d}\bm{\sigma }\cdot \boldsymbol{\xi }%
_{\phi }\,+\,2\,\mu \Big(\int_{\partial \mathcal{B}_{0}}\mathbb{D}(\mathbf{z}%
_{\mathcal{F}})\widetilde{\boldsymbol{n}}\times \boldsymbol{y}\,\mathrm{d}%
\bm{\sigma }\Big)\cdot \mathbf{w}_{\phi } =\int_{\mathcal{F}_{0}}\mathbb{P}%
\boldsymbol{F}_{0}\cdot \boldsymbol{\phi }_{\mathcal{F}}\,\,\mathrm{d}%
\boldsymbol{y}.  \label{C}
\end{align}%
Considering test functions $\boldsymbol{\phi }\in \mathcal{H},$ such that $%
\boldsymbol{\phi }_{\mathcal{B}}=\boldsymbol{0},$\ \ we obtain that there
exists a function $q_{\mathcal{F}}\in L^{2}(0,T;\,H^{1}(\mathcal{F}_{0}))$
satisfying the equation
\begin{equation*}
\mathbf{z}_{\mathcal{F}}^{\prime }\,-\,\mu \Delta \mathbf{z}_{\mathcal{F}%
}\,+\,\nabla q_{\mathcal{F}}=\boldsymbol{F}_{0}\qquad \mathrm{in}\quad
\mathcal{F}_{0}.
\end{equation*}%
Thus for arbitrary $\boldsymbol{\phi }\in \mathcal{H}$, we have
\begin{equation*}
\int_{\mathcal{F}_{0}}\big(\mathbf{z}_{\mathcal{F}}^{\prime }\,-\,\mu \Delta
\mathbf{z}_{\mathcal{F}}\,-\,\boldsymbol{F}_{0}\big)\cdot \boldsymbol{\phi }%
_{\mathcal{F}}\,\,\mathrm{d}\boldsymbol{y}=-\,\int_{\partial \mathcal{B}%
_{0}}q_{\mathcal{F}}\boldsymbol{\phi }_{\mathcal{F}}\cdot \widetilde{%
\boldsymbol{n}}\,\mathrm{d}\bm{\sigma }.
\end{equation*}%
Substituting this equality into \eqref{C}, we obtain that
\begin{align*}
m\Big(\boldsymbol{\xi }^{\prime }(t)\,-\,\frac{\boldsymbol{F}_{1}}{m}\Big)%
\cdot \boldsymbol{\xi }_{\phi }\,& +\,\boldsymbol{I}\Big(\mathbf{w}_{%
\mathcal{B}}^{\prime }\,-\,\boldsymbol{I}^{-1}\boldsymbol{F}_{2}\Big)\cdot
\mathbf{w}_{\boldsymbol{\phi }}\,+\,2\,\mu \,\int_{\partial \mathcal{B}_{0}}%
\mathbb{D}(\mathbf{z}_{\mathcal{F}})\widetilde{\boldsymbol{n}}\,\mathrm{d}%
\bm{\sigma }\cdot \boldsymbol{\xi }_{\phi } \\
& +\,2\,\mu \Big(\int_{\partial \mathcal{B}_{0}}\mathbb{D}(\mathbf{z}_{%
\mathcal{F}})\widetilde{\boldsymbol{n}}\times \boldsymbol{y}\,\mathrm{d}%
\bm{\sigma }\Big)\cdot \mathbf{w}_{\phi }\,=\,\int_{\partial \mathcal{B}%
_{0}}q_{\mathcal{F}}\boldsymbol{\phi }_{\mathcal{F}}\cdot \widetilde{%
\boldsymbol{n}}\,\mathrm{d}\bm{\sigma }.
\end{align*}%
Since the function $\boldsymbol{\phi }$ is divergence free, we have $(%
\boldsymbol{\phi }_{\mathcal{F}}-\boldsymbol{\phi }_{\mathcal{B}})\cdot
\widetilde{\boldsymbol{n}}\mid _{\partial \mathcal{B}_{0}}=0.$ As a
consequence we obtain that
\begin{eqnarray*}
m\boldsymbol{\xi }^{\prime }(t)+\,\int_{\partial \mathcal{B}_{0}}\big(2\mu
\mathbb{D}(\mathbf{z}_{\mathcal{F}})-q_{\mathcal{F}}\mathbb{I}\big)%
\widetilde{\boldsymbol{n}}\,\mathrm{d}\bm{\sigma }\, &=&\,\boldsymbol{F}_{1},
\\
\boldsymbol{I}\mathbf{w}^{\prime }(t)+\int_{\partial \mathcal{B}_{0}}\big(%
2\mu \mathbb{D}(\mathbf{z}_{\mathcal{F}})-q_{\mathcal{F}}\mathbb{I}\big)%
\widetilde{\boldsymbol{n}}\times \boldsymbol{y}\,\mathrm{d}\bm{\sigma }\,
&=&\,\boldsymbol{F}_{2}.
\end{eqnarray*}%
Therefore a problem \eqref{newform} is equivalent to a problem \eqref{lin}.
Finally Propositions \ref{abstractresult} and \ref{analyticityA} imply the
uniqueness of the solution $(\mathbf{z}_{\mathcal{F}},q_{\mathcal{F}},%
\boldsymbol{\xi },\mathbf{w}),$ that satisfies estimate \eqref{est}. $%
\square $
\end{proof}

\subsection{Nonlinear case (Proof of Theorem \protect\ref{th:2.1})}

\label{nonlinearsec} In this section we show Theorem \ref{th:2.1}. To do it
we prove existence and uniqueness results for the modified system %
\eqref{modifiedsystem}. \ The proof is based on the fixed point argument.
Let us define
\begin{equation*}
\mathcal{P}:(\widehat{\mathbf{z}}_{\mathcal{F}},\widehat{q}_{\mathcal{F}},%
\widehat{\boldsymbol{\xi }},\widehat{\mathbf{w}})\rightarrow (\mathbf{z}_{%
\mathcal{F}},q_{\mathcal{F}},\boldsymbol{\xi },\mathbf{w}),
\end{equation*}%
which maps
\begin{equation*}
\mathcal{U}_{T}(\mathcal{F}_{0})\times L^{2}(0,T;H^{1}(\mathcal{F}%
_{0}))\times H^{1}(0,T)\times H^{1}(0,T)
\end{equation*}%
into itself. Functions $(\mathbf{z}_{\mathcal{F}},q_{\mathcal{F}},%
\boldsymbol{\xi },\mathbf{w})=\mathcal{P}(\widehat{\mathbf{z}}_{\mathcal{F}},%
\widehat{q}_{\mathcal{F}},\widehat{\boldsymbol{\xi }},\widehat{\mathbf{w}})$
are the solution of the linear system \eqref{lin}\ with%
\begin{align*}
\boldsymbol{F}_{0}=& \boldsymbol{F}_{0}(\widehat{\mathbf{z}}_{\mathcal{F}},%
\widehat{q}_{\mathcal{F}},\widehat{\boldsymbol{\xi }},\widehat{\mathbf{w}}%
)=-(\mathcal{M}-\mu \mathcal{L}+\mu \Delta )\widehat{\mathbf{z}}_{\mathcal{F}%
}+(\nabla -\mathcal{G})\widehat{q}_{\mathcal{F}}-{\mathcal{N}}\widehat{%
\mathbf{z}}_{\mathcal{F}}+{\widetilde{\boldsymbol{f}}}_{0}, \\
{\boldsymbol{F}_{1}}=& {\boldsymbol{F}_{1}(\widehat{\mathbf{z}}_{\mathcal{F}%
},\widehat{q}_{\mathcal{F}},\widehat{\boldsymbol{\xi }},\widehat{\mathbf{w}}%
)=\widetilde{\boldsymbol{f}}_{1}}+m(\widehat{\mathbf{w}}\times \widehat{%
\boldsymbol{\xi }}) \\
& +\int_{\partial \mathcal{B}_{0}}\mathbb{T}(\widehat{\mathbf{z}}_{\mathcal{F%
}},\widehat{q}_{\mathcal{F}})\widetilde{\boldsymbol{n}}\,\mathrm{d}%
\bm{\sigma }-\,\int_{\partial \mathcal{B}_{0}}\boldsymbol{\mathcal{T}}(%
\widehat{\mathbf{z}}_{\mathcal{F}},\widehat{q}_{\mathcal{F}})\widetilde{%
\boldsymbol{n}}\,\mathrm{d}\bm{\sigma }, \\
{\boldsymbol{F}_{2}}=& {\boldsymbol{F}_{2}}(\widehat{\mathbf{z}}_{\mathcal{F}%
},\widehat{q}_{\mathcal{F}},\widehat{\boldsymbol{\xi }},\widehat{\mathbf{w}})%
{=}\widetilde{\boldsymbol{f}}_{2}+\widehat{\mathbf{w}}\times (\mathbf{I}%
\widehat{\mathbf{w}}) \\
& +\int_{\partial \mathcal{B}_{0}}\boldsymbol{y}\times \mathbb{T}(\widehat{%
\mathbf{z}}_{\mathcal{F}},\widehat{q}_{\mathcal{F}})\widetilde{\boldsymbol{n}%
}\,\mathrm{d}\bm{\sigma }-\,\int_{\partial \mathcal{B}_{0}}\boldsymbol{y}%
\times \boldsymbol{\mathcal{T}}(\widehat{\mathbf{z}}_{\mathcal{F}},\widehat{q%
}_{\mathcal{F}})\widetilde{\boldsymbol{n}}\,\mathrm{d}\bm{\sigma }.
\end{align*}

For some $R>0$ we define the set%
\begin{align*}
K=\{(\widehat{\mathbf{z}}_{\mathcal{F}},\widehat{q}_{\mathcal{F}},\widehat{%
\boldsymbol{\xi }},\widehat{\mathbf{w}})\in & \ \mathcal{U}_{T}(\mathcal{F}%
_{0})\times L^{2}(0,T;H^{1}(\mathcal{F}_{0}))\times H^{1}(0,T)\times
H^{1}(0,T): \\
& \Vert \widehat{\mathbf{z}}_{\mathcal{F}}\Vert _{U_{T}(\mathcal{F}%
_{0})}+\Vert \widehat{q}_{\mathcal{F}}\Vert _{L^{2}(0,T;H^{1}(\mathcal{F}%
_{0}))}+||\boldsymbol{\xi }\Vert _{H^{1}(0,T)}+\Vert \widehat{\mathbf{w}}%
\Vert _{H^{1}(0,T)}\leq R\}.
\end{align*}

As the first step we show that $\mathcal{P}(K)\subset K$. We put $C_{0},B_{0}
$ constants that depends only on $T,$ $\Vert \boldsymbol{u}_{0}\Vert _{H^{1}(%
\mathcal{F}_{0})},\Vert \boldsymbol{u}_{\mathcal{B},0}\Vert _{H^{1}(\mathcal{%
B}_{0})},\Vert \boldsymbol{f}_{0}\Vert _{L_{loc}^{2}(\mathbb{R}^{+};H^{1}(%
\mathcal{F}_{0}))},\Vert (\boldsymbol{f}_{1},\boldsymbol{f}_{2})\Vert
_{L_{loc}^{2}(\mathbb{R}^{+})}$ \ (see the regularity \eqref{R}). \ Moreover
$C_{0},B_{0}$ are nondecreasing functions of $T$. Also $C_{0}$ is a nondecreasing function of $R$. Then Proposition 3.3 gives
\begin{align*}
\Vert \mathbf{z}_{\mathcal{F}}\Vert _{\mathcal{U}_{T}(\mathcal{F}_{0})}&
+\Vert q_{\mathcal{F}}\Vert _{L^{2}(0,T,H^{1}(\mathcal{F}_{0}))}+\Vert
\boldsymbol{\xi }\Vert _{H^{1}(0,T)}+\Vert \mathbf{w}\Vert _{H^{1}(0,T)} \\
& \leq C_{0}(\Vert (\boldsymbol{F}_{1},\boldsymbol{F}_{2})\Vert
_{L^{2}(0,T)}+\Vert \boldsymbol{F}_{0}\Vert _{L^{2}(0,T;L^{2}(\mathcal{F}%
_{0}))}+1).
\end{align*}%
From \cite{T} we have%
\begin{equation*}
\Vert \boldsymbol{F}_{0}\Vert _{L^{2}(0,T;L^{2}(\mathcal{F}_{0}))}+\Vert (%
\boldsymbol{F}_{1},\boldsymbol{F}_{2})\Vert _{L^{2}(0,T)}\leq
C_{0}T^{1/10}+B_{0}.
\end{equation*}%
Therefore it follows that%
\begin{equation*}
\Vert \mathbf{z}_{\mathcal{F}}\Vert _{\mathcal{U}_{T}(\mathcal{F}%
_{0})}+\Vert q_{\mathcal{F}}\Vert _{L^{2}(0,T,H^{1}(\mathcal{F}_{0}))}+\Vert
\boldsymbol{\xi }\Vert _{H^{1}(0,T)}+\Vert \mathbf{w}\Vert _{H^{1}(0,T)}\leq
C_{0}T^{1/10}+B_{0}.
\end{equation*}%
Now choosing $R$ and $T,$ \ such that $4B_{0}<R$\ and $C_{0}(T)T^{1/10}<%
\frac{R}{4},$ \ we deduce that
\begin{equation*}
C_{0}T^{1/10}+B_{0}<R\qquad \text{and}\qquad \mathcal{P}(K)\subset K.
\end{equation*}

In the second step we prove that $\mathcal{P}$ is a contraction operator,
when $T$ is small enough and $R$ is large enough. Let us define
\begin{equation*}
(\mathbf{z}_{\mathcal{F}}^{i},q_{\mathcal{F}}^{i},\boldsymbol{\xi }^{i},%
\mathbf{w}^{i})=\mathcal{P}(\widehat{\mathbf{z}}_{\mathcal{F}}^{i},\widehat{q%
}_{\mathcal{F}}^{i},\widehat{\boldsymbol{\xi }}^{i},\widehat{\mathbf{w}}%
^{i})\qquad \text{for}\quad (\widehat{\mathbf{z}}_{\mathcal{F}}^{i},\widehat{%
q}_{\mathcal{F}}^{i},\widehat{\boldsymbol{\xi }}^{i},\widehat{\mathbf{w}}%
^{i})\in K,\mathcal{\quad }i=1,2,
\end{equation*}%
\ and calculate the diferences%
\begin{eqnarray*}
(\mathbf{z}_{\mathcal{F}},q_{\mathcal{F}},\boldsymbol{\xi },\mathbf{w}) &=&(%
\mathbf{z}_{\mathcal{F}}^{1},q_{\mathcal{F}}^{1},\boldsymbol{\xi }^{1},%
\mathbf{w}^{1})-(\mathbf{z}_{\mathcal{F}}^{1},q_{\mathcal{F}}^{1},%
\boldsymbol{\xi }^{1},\mathbf{w}^{1}), \\
(\widehat{\mathbf{z}}_{\mathcal{F}},\widehat{q}_{\mathcal{F}},\widehat{%
\boldsymbol{\xi }},\widehat{\mathbf{w}}) &=&(\widehat{\mathbf{z}}_{\mathcal{F%
}}^{1},\widehat{q}_{\mathcal{F}}^{1},\widehat{\boldsymbol{\xi }}^{1},%
\widehat{\mathbf{w}}^{1})-(\widehat{\mathbf{z}}_{\mathcal{F}}^{2},\widehat{q}%
_{\mathcal{F}}^{2},\widehat{\boldsymbol{\xi }}^{2},\widehat{\mathbf{w}}^{2}).
\end{eqnarray*}%
Then the functions $(\mathbf{z}_{\mathcal{F}},q_{\mathcal{F}},\boldsymbol{%
\xi },\mathbf{w})$ satisfy the system \eqref{lin} with zero initial
conditions, i.e.
\begin{equation*}
\mathbf{z}_{\mathcal{F}}(0)=\boldsymbol{0}\quad \mathrm{in}\quad \mathcal{F}%
_{0},\qquad \boldsymbol{\xi }(0)=\boldsymbol{0},\qquad \mathbf{w}(0)=%
\boldsymbol{0}
\end{equation*}%
and%
\begin{equation*}
\boldsymbol{F}_{k}=\boldsymbol{F}_{k}(\widehat{\mathbf{z}}_{\mathcal{F}}^{1},%
\widehat{q}_{\mathcal{F}}^{1},\widehat{\boldsymbol{\xi }}^{1},\widehat{%
\mathbf{w}}^{1})-\boldsymbol{F}_{k}(\widehat{\mathbf{z}}_{\mathcal{F}}^{2},%
\widehat{q}_{\mathcal{F}}^{2},\widehat{\boldsymbol{\xi }}^{2},\widehat{%
\mathbf{w}}^{2}),\qquad k=0,1,2.
\end{equation*}%
It is easy to check%
\begin{align*}
\Vert \boldsymbol{F}_{0}\Vert _{L^{2}(0,T;L^{2}(\mathcal{F}_{0}))}& +\Vert (%
\boldsymbol{F}_{1},\boldsymbol{F}_{2})\Vert _{L^{2}(0,T)} \\
& \leq C_{0}T^{1/10}(\Vert \widehat{\mathbf{z}}_{\mathcal{F}}\Vert _{%
\mathcal{U}_{T}(\mathcal{F}_{0})}+\Vert \widehat{q}_{\mathcal{F}}\Vert
_{L^{2}(0,T;H^{1}(\mathcal{F}_{0}))}+\Vert (\widehat{\boldsymbol{\xi }},%
\widehat{\mathbf{w}})\Vert _{H^{1}(0,T)}).
\end{align*}%
Applying Proposition 3.3 we obtain%
\begin{align*}
\Vert \mathbf{z}_{\mathcal{F}}\Vert _{\mathcal{U}_{T}(\mathcal{F}_{0})}&
+\Vert q_{\mathcal{F}}\Vert _{L^{2}(0,T;H^{1}(\mathcal{F}_{0}))}+\Vert
\boldsymbol{\xi }\Vert _{H^{1}(0,T)}+||\mathbf{w}\Vert _{H^{1}(0,T)} \\
& \leq C_{0}T^{1/10}\left( \Vert \widehat{\mathbf{z}}_{\mathcal{F}}\Vert _{%
\mathcal{U}_{T}(\mathcal{F}_{0})}+\ \Vert \widehat{q}_{\mathcal{F}}\Vert
_{L^{2}(0,T;H^{1}(\mathcal{F}_{0})}+\Vert (\widehat{\boldsymbol{\xi }},%
\widehat{\mathbf{w}})\Vert _{H^{1}(0,T)}\right) .
\end{align*}%
Thus, when $T$ is small enough, $\mathcal{P}$ is a contraction operator,
such that the unique fixed point of $\mathcal{P}$ \ is a unique solution $(%
\widetilde{\boldsymbol{u}}_{\mathcal{F}},\widetilde{p}_{\mathcal{F}},%
\widetilde{\boldsymbol{\eta }},\widetilde{\boldsymbol{\omega }})$ of system (%
\ref{modifiedsystem}) in K. For given two strong solutions of (\ref%
{modifiedsystem}), there exists a large enough $R,$ such that these
solutions belong to the set $K.$ Since the system (\ref{modifiedsystem}) has
a unique solution in $K$ by the continuity argument we get that system (\ref%
{fluidmotion})-(\ref{bodymotion}) has a unique solution.

\bigskip

\noindent \textbf{Acknowledgements.} \textit{The work of H. Al Baba and \v{S}%
. Ne\v{c}asová was supported by grant No. 16-03230S of GA\v{C}R in the
framework of RVO 67985840. The work of B. Muha was supported by Croatian
Science Foundation grant number 9477.}


\bigskip

\begin{thebibliography}{99}
\bibitem{CN} N. Chemetov, \v{S}. Ne\v{c}asová, {The motion of the rigid body
in the viscous fluid including collisions. Global solvability result,} %
\newblock Nonlinear Anal. Real World Appl. \textbf{34} (2017), 416--445.

\bibitem{CST} C.~ Conca, J.~ San Martin, M.~ Tucsnak, \newblock Existence of
solutions for the equations modelling the motion of a rigid body in a
viscous fluid, \newblock Commun. Partial Differential Equation \textbf{25}
(2000), 1019--1042.

\bibitem{DEES1} B. Desjardins, M.J. Esteban, \newblock Existence of weak
solutions for the motion of rigid bodies in a viscous fluid, \newblock Arch.
Rational Mech. Anal. \textbf{146} (1999), 59--71.

\bibitem{DEES2} B. Desjardins, M.~J. Esteban, \newblock On weak solutions
for fluid-rigid structure interaction: Compressible and incompressible
models, \newblock Commun. Partial Differential Equations \textbf{25} (2000),
1399--1413.

\bibitem{GH2} D. Gérard-Varet, M. Hillairet, \newblock Existence of weak
solutions up to collision for viscous fluid-solid systems with slip, %
\newblock Comm. Pure Appl. Math. \textbf{67} (2014), no. 12, 2022--2075.

\bibitem{GHC} D. Gérard--Varet, M. Hillairet, C. Wang, \newblock The
influence of boundary conditions on the contact problem in a 3D
Navier-Stokes flow, \newblock J. Math. Pures Appl. (9), \textbf{103} (2015),
no. 1, 1--38.

\bibitem{HES} T.~I. Hesla, \newblock Collision of smooth bodies in a viscous
fluid: A mathematical investigation, \newblock PhD Thesis -- Minnesota, 2005.

\bibitem{HIL} M. Hillairet, \newblock Lack of collision between solid bodies
in a 2{D} incompressible viscous flow, \newblock Comm. Partial Differential
Equations \textbf{32} (2007), no. 7-9, 1345--1371.

\bibitem{HOST} K.-H. Hoffmann, V.~N. Starovoitov, \newblock On a motion of a
solid body in a viscous fluid. Two dimensional case, \newblock Adv. Math.
Sci. Appl. \textbf{9} (1999), 633--648.

\bibitem{IW} A. Inoue, M. Wakimoto, On existence of solutions of the
Navier-Stokes equation in a time dependent domain, \newblock J. Fac. Sci.
Univ. Tokyo Sect. IA Math. \textbf{24} (1977), no. 2, 303--319.

\bibitem{G2} G.P. Galdi, \newblock On the motion of a rigid body in a
viscous liquid: A mathematical analysis with applications, \newblock %
Handbook of Mathematical Fluid Dynamics, Vol. 1, Ed. by Friedlander, D.
Serre, Elsevier, 2002.

\bibitem{S} M.D. Gunzburger, H. Lee, G. Seregin, \newblock Global existence
of weak solutions for viscous incompressible flows around a moving rigid
body in three dimensions, \newblock J. Math. Fluid Mech. {\textbf{2}},
(2000), no. 3, 219--266.

\bibitem{K1} T. Kato, \newblock Fractional powers of dissipative operators, %
\newblock  J. Math. Soc. Japan \textbf{13}, 246--274, (1961)

\bibitem{K2} T. Kato, \newblock Abstract evolution equations of parabolic
type in Banach and Hilbert spaces. \newblock  Nagoya Math. J. \textbf{19},
93--125, (1961)

\bibitem{NP1} J. Neustupa, P. Penel, \newblock Existence of a weak solution
to the Navier-Stokes equation with Navier's boundary condition around
striking bodies, \newblock Comptes Rendus Mathematique \textbf{347} (2009),
no. 11-12, 685--690.

\bibitem{NP2} J. Neustupa, P. Penel, \newblock A Weak solvability of the
Navier-Stokes equation with Navier's boundary condition around a ball
striking the wall, \newblock In the book: Advances in Mathematical Fluid
Mechanics: Dedicated to Giovanni Paolo Galdi, Springer--Verlag Berlin,
(2010) 385--408.

\bibitem{ShibataShimada} Y. Shibata, R. Shimada, \newblock On a generalized
resolvent estimate for the Stokes system with Robin boundary condition, %
\newblock J. Math. Soc. Japan \textbf{59} (2007), no. 2, 469--519.

\bibitem{T} T. Takahashi, \newblock Analysis of strong solutions for the
equations modeling the motion of a rigid-fluid system in a bounded domain, %
\newblock Adv. Differential Equations \textbf{8} (2003), no. 12, 1499--1532.

\bibitem{TT} T. Takahashi, M. Tucsnak, \newblock Marius Global strong
solutions for the two-dimensional motion of an infinite cylinder in a
viscous fluid, \newblock J. Math. Fluid Mech. \textbf{6} (2004), no. 1,
53--77.

\bibitem{Wa} C. Wang, \newblock Strong solutions for the fluid-solid systems
in a 2-D domain, \newblock Asymptot. Anal. \textbf{89} (2014), no. 3-4,
263--306.
\end{thebibliography}
\end{document}